\documentclass[a4paper,12pt]{article}

\usepackage[left=2cm,right=2cm, top=2cm,bottom=2cm,bindingoffset=0cm]{geometry}

\usepackage{verbatim}
\usepackage{amsmath}
\usepackage{amsthm}
\usepackage{amssymb}
\usepackage{delarray}
\usepackage{cite}
\usepackage{hyperref}
\usepackage{mathrsfs}
\usepackage{tikz}
\usetikzlibrary{patterns}
\usepackage{caption}
\DeclareCaptionLabelSeparator{dot}{. }
\newcommand{\be}{\beta}
\newcommand{\ga}{\gamma}

\newcommand{\la}{\lambda}
\newcommand{\om}{\omega}

\newcommand{\eps}{\varepsilon}
\newcommand{\vv}{\varphi}
\newcommand{\iy}{\infty}
\theoremstyle{plain}

\newtheorem{thm}{Theorem}
\newtheorem{lem}{Lemma}

\theoremstyle{definition}

\newtheorem{example}{Example}
\newtheorem{alg}{Algorithm}
\newtheorem{ip}{Inverse Problem}

\theoremstyle{remark}

 \allowdisplaybreaks

\begin{document}

\begin{center}
{\large\bf Local solvability and stability of inverse problems for Sturm-Liouville operators with a discontinuity}
\\[0.2cm]
{\bf Chuan-Fu Yang, Natalia P. Bondarenko} \\[0.2cm]
\end{center}

\vspace{0.5cm}

{\bf Abstract.} Partial inverse problems are studied for Sturm-Liouville operators with a discontinuity.
The main results of the paper are local solvability and stability of the considered inverse problems. Our approach is based on
a constructive algorithm for solving the inverse problems. The key role in our method is played by the Riesz-basis property of a special vector-functional system
in a Hilbert space. In addition, we obtain a new uniqueness theorem for recovering the potential on a part of the interval,
by using a fractional part of the spectrum.

\medskip


\medskip


\vspace{1cm}

{\large \bf 1. Introduction}

\bigskip

This paper concerns inverse spectral theory for the Sturm-Liouville operator with a discontinuity.
We focus on the following boundary value problem:
\begin{equation}\label{eqv}
  -y''+q(x)y=\lambda^2 y, \ 0<x<1,
\end{equation}
with the boundary conditions
\begin{equation}\label{bc}
  y'(0)-h_1y(0)=0=y'(1)+h_2y(1)
\end{equation}
and with the jump conditions
\begin{equation}\label{jc}
  y(d+0)=a_1y(d-0), \ y'(d+0)=a^{-1}_{1}y'(d-0)+a_2y(d-0),
\end{equation}
where $\la^2$ is the spectral parameter; $q\in L_2(0,1)$ is a real-valued function, called the {\it potential};
$0<d\leq 1/2$ is the discontinuity position; $a_j$ and $h_j$ are real numbers for $j = 1, 2$, $a_1 > 0$.

Inverse problems of spectral analysis consist in recovering operators from their spectral characteristics.
In recent years, inverse spectral problems attract much attention of mathematicians because of applications in quantum mechanics,
chemistry, acoustics, nanotechnology and other fields of science and engineering (see, e.g., \cite{FY01} and references therein).
In particular, the eigenvalue problem~\eqref{eqv}-\eqref{jc} with a discontinuity arises in geophysical models for oscillations of the Earth
(see \cite{And97, LU81}) and in electronics for constructing parameters of heterogeneous electronic lines with desirable technical characteristics	
(see \cite{LS64, MF80}).

The classical results of inverse problem theory were obtained for Sturm-Liouville operators without discontinuities
(see the monographs \cite{FY01, Mar77, Lev84, PT87}). Operators with discontinuities are more difficult for investigation,
so the complete inverse problem theory for such operators has not been constructed yet. The majority of the papers on inverse problems
for discontinuous operators concern uniqueness theorems (see \cite{Hald84, SY08, OK12, Yang14-1, Wang15}). The most extensive
study of recovering the problem~\eqref{eqv}-\eqref{jc} from the spectral data has been provided in \cite[Section~4.4]{FY01}.
The authors of~\cite{FY01} have developed a constructive procedure for solving the inverse problem, by using the method
of spectral mappings. They have also obtained necessary and sufficient conditions for solvability of the inverse problem.
We also mention the papers \cite{AKM96, Shep94}, concerning inverse scattering for discontinuous operators on the line
and on the half-line, respectively.

The present paper deals with {\it partial} inverse problems. In literature, such problems are also called half-inverse problems and inverse problems
by mixed data. Partial inverse problems consist in recovering operators
in the case, when partial information on their coefficients is given a priori. Indeed, for solving partial inverse problems,
usually less spectral data are required, than for complete problems.
The first half-inverse problem was studied
by Hochstadt and Lieberman for the Sturm-Liouville operator without discontinuities on a finite interval (see \cite{HL78}).
It has been proved in \cite{HL78}, that, if the potential is known a priori on a half-interval, the potential on the other
half-interval is uniquely specified by one spectrum. Further generalizations and developments of this result were obtained in
\cite{GS00, Sakh01, Hor01, HM04, Piv12, SBI11, But11}. In the recent paper \cite{BY19}, we have investigated half-inverse problems
for the Sturm-Liouville operator with a discontinuity \eqref{eqv}-\eqref{jc}. Reconstruction algorithms for solving these problems
have been developed, and conditions for existence of solutions have been obtained. However, the questions of local solvability and
stability for discontinuous partial inverse problems are still open, so we devote our present paper to these issues.
Note that stability plays an important role in justification of numerical methods for solving various problems
of mathematical physics.

Let us provide rigorous formulations of the studied problems and the main results.
For convenience, we rewrite the problem \eqref{eqv}-\eqref{jc} in the following form:
\begin{align}
   &-y''_j+q_j(x)y_j=\lambda^2 y_j, \quad x\in (0, d_j), \, j=1,2,  \label{2.1}\\
   &y'_j(0)-h_jy_j(0)=0, \label{2.2} \\
   &y_2(d_2)=a_1y_1(d_1),  \label{2.3}\\
   &y'_2(d_2)+a^{-1}_1y'_1(d_1)+a_2y_1(d_1)=0.\label{2.4}
\end{align}
Here $d_1 = d$, $d_2 = 1-d$, $q_j\in L_2(0,d_j)$ for $j=1,2$, and $q_1(x):=q(x)|_{[0,d_1]}$, $q_2(x):=q(1-x)|_{[0,d_2]}$.
Denote the boundary value problem \eqref{2.1}-\eqref{2.4} by $L = L(d, q_1, q_2, h_1, h_2, a_1, a_2)$.

The spectrum of the problem $L$ is a countable set of real and simple eigenvalues $\{ \la_n^2 \}_{n \ge 0}$,
$\la_n^2 < \la_{n+1}^2$ for $n \ge 0$. Without loss of generality, we assume that $\la_n > 0$ for all $n \ge 0$.
One can achieve this condition by a shift of the spectrum. First we will study the following inverse problem.

\begin{ip} \label{ip:1}
Suppose that $d = 1/2$. Given the eigenvalues $\{ \la_n^2 \}_{n \ge 0}$, the potential $q(x)$ for $x \in (1/2, 1)$ and the coefficients $h_2$, $a_1$,
construct $q(x)$ for $x \in (0, 1/2)$, $h_1$ and $a_2$.
\end{ip}

The uniqueness of solution for Inverse Problem~\ref{ip:1} was proved by Hald (see \cite{Hald84}).
In \cite{BY19} the reconstruction algorithm has been provided and conditions have been obtained, necessary and sufficient
for existence of solution. In this paper, we investigate the issues of local solvability and stability for Inverse Problem~\ref{ip:1}.

Suppose that we have a fixed problem $L = L(1/2, q_1, q_2, h_1, h_2, a_1, a_2)$ and its spectrum $\{ \la_n^2 \}_{n \ge 0}$.
The sequence $\Lambda := \{ \la_n \}_{n \ge 0}$ satisfies the following asymptotic relation (see \cite{Yang14}):
\begin{equation} \label{asymptla}
  \lambda_n=n\pi+\frac{(-1)^n a+b}{n\pi}+\frac{\beta_n}{n},
\end{equation}
where $a$ and $b$ are constants, depending on the problem $L$ (see Section~2 for details).
Here and below the same notation $\{ \be_n \}$ is used for various sequences from $l_2$.
We introduce for the fixed numbers $a$ and $b$ the following space of real sequences:
$$
    \mathbb M = \left\{ \mathcal M = \{ \mu_n \}_{n \ge 0} \colon \mu_n = \pi n + \frac{(-1)^n a + b}{\pi n} + \frac{\be_n}{n}, \: \{ \be_n \} \in l_2 \right\}.
$$
Define the distance on $\mathbb M$ as follows:
$$
    \rho(\mathcal M, \tilde {\mathcal M}) =\left(  \sum_{n = 0}^{\iy} (n + 1)^2(\mu_n - \tilde \mu_n)^2 \right)^{1/2},
    \quad \mathcal M = \{ \mu_n \}_{n \ge 0}, \: \tilde {\mathcal M} = \{ \tilde \mu_n \}_{n \ge 0}, \quad
     \mathcal M, \tilde {\mathcal M} \in \mathbb M.
$$

The following theorem establishes local solvability and stability of Inverse Problem~\ref{ip:1}.

\begin{thm} \label{thm:loc}
For every problem $L = L(1/2, q_1, q_2, h_1, h_2, a_1, a_2)$ there exists $\eps > 0$, such that for any sequence
$\tilde \Lambda = \{ \tilde \la_n \}_{n \ge 0} \in \mathbb M$, satisfying the estimate $\rho(\Lambda, \tilde \Lambda) \le \eps$,
there exist a function $\tilde q_1 \in L_2(0, 1/2)$ and a number $\tilde h_1 \in \mathbb R$, such that the problem
$\tilde L = L(1/2, \tilde q_1, q_2, \tilde h_1, h_2, a_1, a_2)$ has the spectrum $\{\tilde \la_n^2 \}_{n \ge 0}$
and
\begin{equation} \label{estq}
   \| q_1 - \tilde q_1 \|_{L_2} \le C \rho(\Lambda, \tilde \Lambda), \quad |h_1 - \tilde h_1| \le C \rho(\Lambda, \tilde \Lambda),
\end{equation}
where the constant $C > 0$ depends only on the problem $L$ and does not depend on $\tilde \Lambda$.
\end{thm}

Comparing with the existence theorem in \cite{BY19}, Theorem~\ref{thm:loc} has such an advantage, that
it contains only one simple requirement.
Namely, the numbers $\{ \tilde \la_n^2 \}_{n \ge 0}$ have to be ``sufficiently close'' in some sense to the eigenvalues $\{ \la_n^2 \}_{n \ge 0}$
of the known problem $L$. The conditions of Theorem~2.1 in \cite{BY19} are more complicated. However,
Theorem~2.1 in~\cite{BY19} establishes the global solvability of Inverse Problem~\ref{ip:1}, while Theorem~\ref{thm:loc} has a local nature.

In order to prove Theorem~\ref{thm:loc}, we develop a new algorithm for solving Inverse Problem~\ref{ip:1}. Our method
is based on construction of a special Riesz basis in an appropriate Hilbert space of vector-functions.
Analogous ideas have been used to prove local solvability and stability of a partial inverse problem for the Sturm-Liouville operator on a graph
\cite{Bond18} and of the inverse transmission eigenvalue problem \cite{BB17}.
Basic information about Riesz bases can be found in \cite[Section~1.8.5]{FY01}.

Our approach can be easily generalized to the case
$0 < d < 1/2$. In this case, one can recover the potential $q_1(x)$ and the coefficient $h_1$ from some
fractional part of the spectrum.
Let $I$ be a fixed infinite subset of $\mathbb N \cup \{ 0 \}$. Consider the following partial inverse problem.

\begin{ip} \label{ip:2}
Suppose that $0 < d < 1/2$. Given the eigenvalues $\{ \la_n^2 \}_{n \in I}$, the potential $q(x)$ for $x \in (d, 1)$ and the
coefficients $h_2$, $a_1$, $a_2$, $\om_1 := h_1 + \int_0^d q_1(x) \,dx$, construct $q(x)$ for $x \in (0, d)$ and $h_1$.
\end{ip}

Uniqueness and existence of solution of Inverse Problem~\ref{ip:2} certainly depend on the discontinuity position $d$ and on the set $I$.
Note that, in some special cases, the numbers $a_2$ and $\om_1$ can be determined from eigenvalue asymptotics
(in particular, when $d = 1/2$, $I = \mathbb N \cup \{ 0 \}$).

In this paper, we prove the uniqueness theorem and develop a constructive algorithm for solving Inverse Problem~\ref{ip:2}
under some conditions. Local solvability and stability of this problem are also established.

The paper is organized as follows. In Section~2, the main equation is derived for solving Inverse Problems~\ref{ip:1}
and~\ref{ip:2}. Using this equation, we obtain reconstruction algorithm for Inverse Problem~\ref{ip:1} for $d = 1/2$.
Section~3 is devoted  to the proof of Theorem~\ref{thm:loc}. In Section~4, we study Inverse Problem~\ref{ip:2} (the case $0 < d < 1/2$).
The uniqueness theorem is proved, constructive solution is provided, local solvability and stability are obtained for this case.
Since the general scheme of the proof is similar to the case $d = 1/2$, we focus our attention on the result formulations
and the most important differences, and do not elaborate into details. In conclusion of Section~4, we consider
an example $d = 1/4$.

\bigskip

{\large \bf 2. Constructive solution}

\bigskip

In this section, the main equation for Inverse Problems~\ref{ip:1} and~\ref{ip:2} is derived. A crucial role in our analysis
is played by construction of a special vector-functional system in a Hilbert space.
Further we focus on the case $d = 1/2$, corresponding to Inverse Problem~\ref{ip:1}. For this case, we prove that the constructed system
of vector-functions is a Riesz basis. Finally, we obtain Algorithm~\ref{alg:1} for solution of the inverse problem.

Let us start with some notations.
For $j = 1, 2$, denote by $\vv_j(x, \la)$ the solution of equation~\eqref{2.1}, satisfying the initial conditions
$\vv_j(0, \la) = 1$, $\vv'_j(0, \la) = h_j$.
For any $c > 0$, let $\mathcal{L}^{c}$ be the class of entire functions of exponential type not greater than $c$,
belonging to $L_2(\mathbb{R})$ for real $\la$.
According to the results of~\cite{FY01, Bond18}, the following relations hold for $j = 1, 2$:
\begin{align}\label{2.5}
  & \varphi_j\left(d_j,\lambda\right)=\cos (\la d_j)+\om_j\frac{\sin (\la d_j)}{\lambda}+\frac{\psi^{(j)}_1(\lambda)}{\lambda}, \\
    \label{2.6}
  & \varphi'_j\left(d_j,\lambda\right)=-\lambda\sin (\la d_j) +\om_j \cos (\la d_j) +\psi^{(j)}_2(\lambda),
\end{align}
where $\om_j = h_j + \frac{1}{2} \int_0^{d_j} q_j(x) \, dx$, and the functions
$\psi^{(j)}_k$, $j, k = 1, 2$, belong to the class $\mathcal{L}^{d_j}$. Note that
$\vv_j(x, \la)$ and $\vv_j'(x, \la)$ are even functions of $\la$. Consequently, the function $\psi_1^{(j)}(\la)$ is odd
and the function $\psi_2^{(j)}(\la)$ is even for $j = 1, 2$.
Therefore the functions $\psi_j^{(1)}(\la)$ can be represented in the following form:
\begin{equation} \label{intpsi}
\psi_1^{(1)}(\la) = \int_0^{d} K_1(x) \sin (\la x) \, dx, \quad
\psi_2^{(1)}(\la) = \int_0^{d} K_2(x) \cos (\la x) \, dx, \quad K_j \in L_2(0, d), \: j = 1, 2.
\end{equation}

The eigenvalues $\{ \la_n^2 \}_{n \ge 0}$ of the boundary value problem $L$
coincide with the squared zeros of the characteristic function:
\begin{equation} \label{char}
    \Delta(\la) = a_1 \vv_1(d_1, \la) \vv_2'(d_2, \la) + a_1^{-1} \vv_1'(d_1, \la) \vv_2(d_2, \la) + a_2 \vv_1(d_1, \la) \vv_2(d_2, \la).
\end{equation}

Substituting \eqref{2.5}, \eqref{2.6} and \eqref{intpsi} into \eqref{char} and taking $\la = \la_n$, we arrive at the relation
\begin{multline} \label{long}
    \left( \la_n \cos (\la_n d) + \om_1 \sin (\la_n d) +
\int_0^{d} K_1(x) \sin (\la_n x) \, dx \right) \cdot \frac{1}{\la_n} (a_1 \vv_2'(1-d, \la_n) \\ + a_2 \vv_2(1-d, \la_n))
+ \left(-\la_n \sin (\la_n d) + \om_1 \cos (\la_n d) + \int_0^d K_2(x) \cos (\la_n x) \, dx \right)
\\ \cdot a_1^{-1} \vv_2(1-d, \la_n) = 0, \quad n \ge 0.
\end{multline}

Introduce the Hilbert space $H = L_2(0, d) \oplus L_2(0, d)$ of real-valued vector-functions
$w = \begin{pmatrix} w_1 \\ w_2 \end{pmatrix}$, $w_j \in L_2(0, d)$, $j = 1, 2$. The scalar product and the norm in $H$
are defined as follows:
\begin{gather*}
    (g, w)_H = \int_0^{d} (g_1(x) w_1(x) + g_2(x) w_2(x)) \, dx, \quad
    \| w \|_H^2 = \int_0^{d} (w_1^2(x) + w_2^2(x)) \, dx, \\ g = \begin{pmatrix} g_1 \\ g_2 \end{pmatrix}, \:
     w = \begin{pmatrix} w_1 \\ w_2 \end{pmatrix},
    \quad g, w \in H.
\end{gather*}

Recall that $\la_n > 0$ for all $n \ge 0$.
Clearly, the vector-functions
\begin{gather} \nonumber
    K(x) := \begin{pmatrix} K_1(x) \\ K_2(x) \end{pmatrix}, \\ \label{defv}
    v_n(x) := \begin{pmatrix} \frac{1}{\la_n} (a_1 \vv_2'(1-d,  \la_n) + a_2 \vv_2(1-d, \la_n)) \sin (\la_n x) \\
                             a_1^{-1} \vv_2(1-d, \la_n) \cos (\la_n x) \end{pmatrix}, \: n \ge 0,
\end{gather}
belong to $H$, and the relation~\eqref{long} can be rewritten in the form
\begin{equation} \label{scal}
    (K, v_n)_H = f_n, \quad n \ge 0,
\end{equation}
where
\begin{multline} \label{deff}
    f_n = -\frac{1}{\la_n} (a_1 \vv_2'(1-d,  \la_n) + a_2 \vv_2(1-d, \la_n)) \left( \la_n \cos (\la_n d) +
    \om_1 \sin (\la_n d) \right) \\
     + a_1^{-1} \vv_2(1-d, \la_n) \left(-\la_n \sin (\la_n d) + \om_1 \cos (\la_n d) \right).
\end{multline}

We call the relation~\eqref{scal} {\it the main equation} of Inverse Problems~\ref{ip:1} and~\ref{ip:2}.

Let us first consider the case $d = 1/2$.
Suppose that the eigenvalues $\{ \la_n^2 \}_{n \ge 0}$ together with the potential $q_2$ and the coefficients
$a_1$, $h_2$ are given, and we have to find $q_1$, $h_1$ and $a_2$.
Since $q_2$ and $h_2$ are known, one can easily construct the functions $\vv_2(1/2, \la)$, $\vv_2'(1/2, \la)$
and the constant $\om_2$.

Recall that, for $d = 1/2$, the eigenvalues satisfy the asymptotic formula~\eqref{asymptla},
where (see \cite{Yang14}):
\begin{align} \label{defa}
&  a=\frac{a_2}{a_1+a_1^{-1}}+\frac{a_1-a_1^{-1}}{a_1+a_1^{-1}}(\om_2 - \om_1), \\ \label{defb}
&  b=\frac{a_2}{a_1+a_1^{-1}}+\om_1 + \om_2,
\end{align}
By using the eigenvalues~$\{ \la_n^2 \}_{n \ge 0}$, one can determine the constants $a$ and $b$,
and then find $\om_1$ and $a_2$, solving the linear system~\eqref{defa}-\eqref{defb}.
Now one can construct $\{ v_n \}_{n \ge 0}$ and $\{ f_n \}_{n \ge 0}$, by using the formulas~\eqref{defv} and~\eqref{deff},
respectively.

\begin{thm} \label{thm:Riesz}
The system of vector-functions $\{ v_n \}_{n \ge 0}$, defined by~\eqref{defv}, is a Riesz basis in $H$.
\end{thm}

The proof of Theorem~\ref{thm:Riesz} will rely on the two following Lemmas.

\begin{lem} \label{lem:complete}
The system $\{ v_n \}_{n \ge 0}$ is complete in $H$.
\end{lem}

\begin{proof}
Suppose that, on the contrary, the system $\{ v_n \}_{n \ge 0}$ is not complete in $H$. Then there exists
a nonzero element $w \in H$, such that $(w, v_n)_H = 0$, $n \ge 0$. In other words, there exist functions
$w_1$ and $w_2$ from $L_2(0, 1/2)$, such that
$$
    \int_0^{1/2}\!\!\left(\!w_1(x) \frac{1}{\la_n} (a_1 \vv_2'(1/2, \la_n) \!+\! a_2 \vv_2(1/2, \la_n)) \sin (\la_n x) \!+
    \!w_2(x) a_1^{-1} \vv_2(1/2, \la_n) \cos (\la_n x) \!\right) \, \!dx \!=\! 0,
$$
for all $n \ge 0$. Hence the entire function
\begin{multline} \label{defH}
    W(\la) := \int_0^{1/2} \biggl(w_1(x) \frac{1}{\la} (a_1 \vv_2'(1/2, \la) + a_2 \vv_2(1/2, \la)) \sin (\la x) \\ +
    w_2(x) a_1^{-1} \vv_2(1/2, \la) \cos (\la x) \biggr) \, dx
\end{multline}
has the zeros $\{ \pm \la_n \}_{n \ge 0}$.
Consequently, the function $\dfrac{W(\la)}{\Delta(\la)}$ is entire. By virtue of \eqref{2.5}, \eqref{2.6}, \eqref{char}
and~\eqref{defH}, the following estimates hold:
$$
    |W(\la)| \le C \exp(|\mbox{Im}\,\la|), \quad |\Delta(\la)|\ge C |\la|\exp(|\mbox{Im}\,\la|),
$$
for $|\la| > \la^*$, $\eps < |\arg \la| < \pi - \eps$, where $\la^*$ and $\eps$ are some positive numbers.
Hence
$$
    \left| \frac{W(\la)}{\Delta(\la)}\right| \le \frac{C}{|\la|}
$$
for the same values of $\la$. Applying Phragmen-Lindel\"of's theorem \cite{BFY} and Liouville's theorem to the function
$\dfrac{W(\la)}{\Delta(\la)}$, we conclude that $W(\la) \equiv 0$. This together with~\eqref{defH} imply that $w_1 = w_2 = 0$
in $L_2(0, 1/2)$, so $w = 0$ in $H$. Thus, we have arrived at the contradiction, that proves the completeness of the system
$\{ v_n \}_{n \ge 0}$.
\end{proof}

\begin{lem} \label{lem:asymptv}
The following asymptotic relation holds:
$$
    v_n(x) = v_n^0(x) + O\left( n^{-1} \right), \quad n \to \iy,
$$
where
$$
    v_{2k}^0(x) = (-1)^k a_1^{-1} \begin{pmatrix}
	                             0 \\ \cos (2 k \pi x)		
                                \end{pmatrix},
    v_{2k+1}^0(x) = (-1)^{k+1} a_1 \begin{pmatrix}
                                     \sin ((2 k + 1) \pi x) \\ 0
                                \end{pmatrix}, \quad k \ge 0,
$$
and the $O$-estimate is uniform with respect to $x \in [0, 1/2]$.
\end{lem}

\begin{proof}[Proof of Theorem~\ref{thm:Riesz}]
By virtue of Lemmas~\ref{lem:complete} and~\ref{lem:asymptv}, the system $\{ v_n \}_{n \ge 0}$ is complete in $H$
and quadratically close to the system $\{ v_n^0 \}_{n \ge 0}$. The latter system, obviously, is a Riesz basis in $H$.
Hence $\{ v_n \}_{n \ge 0}$ is also a Riesz basis.
\end{proof}

Thus, the numbers $\{ f_n \}_{n \ge 0}$, defined by~\eqref{deff}, are the coordinates of the vector-function $K$
with respect to the Riesz basis, orthonormal to $\{ v_n \}_{n \ge 0}$. One can recover $K$ from these coordinates,
and then find $\vv_1(1/2, \la)$ and $\vv_1'(1/2, \la)$ by~\eqref{2.5}, \eqref{2.6} and \eqref{intpsi}.

The function $\dfrac{\vv_1'(1/2, \la)}{\vv_1(1/2, \la)}$ is the Weyl function of the boundary value problem
$$
    -y_1''(x) + q_1(x) y_1(x) = \la^2 y_1(x), \: x \in (0, 1/2), \quad  y'_1(0)-h_1 y_1(0)=0, \quad y_1(1/2) = 0.
$$
Weyl functions and their generelizations are often used in inverse problem theory as natural spectral characteristics
of various differential operators (see \cite{Mar77, FY01}). The potential $q_1(x)$ and the coefficient $h_1$ can be uniquely
recoovered from the Weyl function (see \cite{FY01}).

Thus, we arrive at the following algorithm for solving Inverse Problem~\ref{ip:1}.

\begin{alg} \label{alg:1}
Suppose that $d = 1/2$. Let $\{ \la_n^2 \}_{n \ge 0}$, $q_2$, $a_1$, $h_2$ be given. The potential $q_1$ and the numbers $h_1$, $a_2$ have
to be found.

\begin{enumerate}
\item Construct the functions $\vv_2(1/2, \la)$ and $\vv'_2(1/2, \la)$, by using $q_2$ and $h_2$.
\item Calculate $\om_2 := h_2 + \frac{1}{2} \int_0^{1/2} q_2(x) \, dx$.
\item Find $a$ and $b$ from~\eqref{asymptla}, by using the formulas:
$$
    \ga_n := (\la_n - \pi n) \pi n, \: n \ge 0, \quad
    a = \frac{1}{2} \lim_{n \to \iy} (\ga_{2n} - \ga_{2n+1}), \quad
    b = \frac{1}{2} \lim_{n \to \iy} (\ga_{2n} + \ga_{2n+1}).
$$
\item Solving the system of linear equations~\eqref{defa}-\eqref{defb}, find $\om_1$ and $a_2$:
$$
    \om_1 = -\frac{1}{2 a_1} ((a_1 + a_1^{-1}) (a - b) + 2 a_1^{-1} \om_2), \quad
    a_2 = (b - \om_1 - \om_2)(a_1 + a_1^{-1}).
$$
\item Construct the vector-functions $\{ v_n(x) \}_{n \ge 0}$ and the numbers $\{ f_n \}_{n \ge 0}$, using~\eqref{defv} and~\eqref{deff}, respectively.
\item Determine the vector-function $K \in H$, by using its coordinates $f_n$:
$$
    K(x) = \sum_{n = 0}^{\iy} f_n v_n^*(x),
$$
where $\{ v_n^*(x) \}_{n \ge 0}$ is the basis, biorthonormal to $\{ v_n(x) \}_{n \ge 0}$ in $H$.
\item Using the components $K_1(x)$ and $K_2(x)$ of $K(x)$, construct $\psi_1^{(1)}(\la)$ and $\psi_2^{(1)}(\la)$
by~\eqref{intpsi}, and then calculate $\vv_1(d, \la)$ and $\vv_1'(d, \la)$ by~\eqref{2.5} and~\eqref{2.6}.
\item Recover $q_1$ and $h_1$ from the Weyl function $\dfrac{\vv_1'(d,\la)}{\vv_1(d,\la)}$, solving the classical inverse problem
by the method of spectral mappings (see \cite{FY01}).
\end{enumerate}
\end{alg}

\bigskip

{\large \bf 3. Local solvability and stability}

\bigskip

The goal of this section is to prove Theorem~\ref{thm:loc}.
Suppose that we have a fixed problem $L = L(1/2, q_1, q_2, h_1, h_2, a_1, a_2)$ and its spectrum $\{ \la_n^2 \}_{n \ge 0}$.
Let us apply Algorithm~\ref{alg:1} to the sequence $\{ \tilde \la_n^2 \}_{n \ge 0}$ together with the data
$(q_2, h_2, a_1)$.
For sufficiently small $\eps > 0$, the condition $\rho(\Lambda, \tilde \Lambda) \le \eps$ imples that
the values $\{ \tilde \la_n \}_{n \ge 0}$ are distinct and positive.
If a certain symbol $\ga$ denotes an object, related to the problem $L$, we will denote
by $\tilde \ga$ with tilde the analogous object, constructed by $\{ \tilde \la_n^2 \}_{n \ge 0}$ and $(q_2, h_2, a_1)$. Note that the steps 1-4 of
Algorithm~\ref{alg:1} do not use $\{ \tilde \la_n^2 \}_{n \ge 0}$, so $\tilde \vv_2(1/2, \la) \equiv \vv_2(1/2, \la)$,
$\tilde \vv_2'(1/2, \la) = \vv_2'(1/2, \la)$, $\tilde \om_j = \om_j$, $j = 1, 2$, $\tilde a_2 = a_2$.
Then consider the system of vector-functions $\{ \tilde v_n(x) \}_{n \ge 0}$, defined as follows:
\begin{equation} \label{deftv}
    \tilde v_n(x) := \begin{pmatrix} \frac{1}{\tilde \la_n} (a_1 \vv_2'(1/2,  \tilde \la_n) + a_2 \vv_2(1/2, \tilde \la_n)) \sin (\tilde \la_n x) \\
                             a_1^{-1} \vv_2(1/2, \tilde \la_n) \cos (\tilde \la_n x) \end{pmatrix}, \: n \ge 0,
\end{equation}
and the sequence $\{ \tilde f_n \}_{n \ge 0}$, defined similarly to~\eqref{deff}, but with $\tilde \la_n$ instead of $\la_n$.
\begin{lem}
For every problem $L$, there exists $\eps > 0$, such that for any sequence
$\tilde \Lambda = \{ \tilde \la_n \}_{n \ge 0} \in \mathbb M$, satisfying the estimate $\rho(\Lambda, \tilde \Lambda) \le \eps$,
then the following estimates are valid
\begin{gather} \label{estvn}
    \left( \sum_{n = 0}^{\iy} \| v_n - \tilde v_n \|_H^2 \right)^{1/2} \le C \rho(\Lambda, \tilde \Lambda), \\ \label{estfn}
    \left( \sum_{n = 0}^{\iy} (f_n - \tilde f_n)^2 \right)^{1/2} \le C \rho(\Lambda, \tilde \Lambda).
\end{gather}
\end{lem}

Here and below we use the same symbol $C$ for different constants, that depend only on $L$ and do not depend on
$\{ \tilde \la_n^2 \}_{n \ge 0}$, $n$, $x$, etc.

\begin{proof}
Using the standard approach, based on Schwartz's Lemma (see \cite[Section 1.6.1]{FY01}), we obtain the following estimates
for $n \ge 0$:
\begin{gather*}
    |\sin (\la_n x) - \sin (\tilde \la_n x)| \le C |\la_n - \tilde \la_n|, \quad
    |\cos (\la_n x) - \cos (\tilde \la_n x)| \le C |\la_n - \tilde \la_n|, \quad x \in [0, 1/2], \\
    |\vv_2(1/2, \la_n) - \vv_2(1/2, \tilde \la_n)| \le C |\la_n - \tilde \la_n|, \\
    |\vv_2'(1/2, \la_n) - \vv_2'(1/2, \tilde \la_n)| \le C (n + 1) |\la_n - \tilde \la_n|.
\end{gather*}
Using these estimates together with~\eqref{defv}, \eqref{deff} and~\eqref{deftv}, we arrive at~\eqref{estvn} and~\eqref{estfn}.
\end{proof}

\begin{lem} \label{lem:stabR}
For every problem $L$, there exists $\eps > 0$, such that for any sequence
$\tilde \Lambda = \{ \tilde \la_n \}_{n \ge 0} \in \mathbb M$, satisfying the estimate $\rho(\Lambda, \tilde \Lambda) \le \eps$,
there exists a unique vector-function $\tilde K \in H$, such that $(\tilde K, \tilde v_n) = \tilde f_n$ for all $n \ge 0$.
Moreover, the estimate $\| K - \tilde K \|_{H} \le C \rho(\Lambda, \tilde \Lambda)$ is valid.
\end{lem}

\begin{proof}
By virtue of Theorem~\ref{thm:Riesz}, the system $\{ v_n \}_{n \ge 0}$ is a Riesz basis in $H$.
In view of~\eqref{estvn}, for sufficiently small $\eps > 0$, the system $\{ \tilde v_n \}_{n \ge 0}$
is also a Riesz basis.
Since $\{ f_n \}_{n \ge 0}$ are the coordinates of $K \in H$ with respect to the Riesz basis $\{ v_n^* \}_{n \ge 0}$,
we have $\{ f_n \}_{n \ge 0} \in l_2$. The estimate~\eqref{estfn} implies that $\{ \tilde f_n \}_{n \ge 0}\in l_2$.
Then there exists a unique vector-function $\tilde K \in H$ with the coordinates $\{ \tilde f_n \}_{n \ge 0}$
with respect to the Riesz basis $\{ \tilde v_n^* \}_{n \ge 0}$, i.e. $(\tilde K, \tilde v_n) = \tilde f_n$, $n \ge 0$.
By virtue of \cite[Lemma~5]{Bond18}, the following estimate holds: $\| K - \tilde K \|_{H} \le C \rho(\Lambda, \tilde \Lambda)$,
so the proof is finished.
\end{proof}

Using the components $\tilde K_j(x)$, $j = 1, 2$, of the vector-function $\tilde K$, we construct the entire functions
$\tilde \eta_j(\la)$, $j = 1, 2$, as follows:
\begin{align} \label{eta1}
  & \tilde \eta_1(\la)=\cos \frac{\la}{2} +\om_1\frac{\sin \frac{\la}{2}}{\lambda}+
  \frac{1}{\lambda} \int_0^{1/2} \tilde K_1(x) \sin (\la x) \, dx , \\    \label{eta2}
  & \tilde \eta_2(\la)=-\lambda\sin \frac{\la}{2} +\om_1 \cos \frac{\la}{2} + \int_0^{1/2} \tilde K_2(x) \cos (\la x) \, dx.
\end{align}

Denote by $\{ \pm \mu_n \}_{n \ge 0}$, $\{ \pm \nu_n \}_{n \ge 0}$, $\{ \pm \tilde \mu_n \}_{n \ge 0}$ and
$\{ \pm \tilde \nu_n \}_{n \ge 0}$ the zeros of $\vv_1(1/2, \la)$, $\vv_1'(1/2, \la)$, $\tilde \eta_1(\la)$
and $\tilde \eta_2(\la)$, respectively. These functions are even, so if any of them has a zero $\la$, it also has
the zero $(-\la)$.

Note that $\{ \mu_n^2 \}_{n \ge 0}$ and $\{ \nu_n^2 \}_{n \ge 0}$ are the spectra of the following boundary value problems
$L_j$:
\begin{equation} \label{Lj}
    -y_1''(x) + q_1(x) y_1(x) = \la^2 y_1(x), \: x \in (0, 1/2), \quad  y'_1(0)-h_1 y_1(0)=0, \quad y_1^{(j)}(1/2) = 0,
\end{equation}
for $j = 0, 1$, respectively. Hence the eigenvalues $\{ \mu_n^2 \}_{n \ge 0}$ and $\{ \nu_n^2 \}_{n \ge 0}$ are real, simple
and have appropriate asymptotics (see \cite{FY01} for details).
Using the estimates $\| K_j - \tilde K_j \|_{L_2} \le C \rho(\Lambda, \tilde \Lambda)$, $j = 1, 2$,
one can obtain the following result.

\begin{lem} \label{lem:zeros}
For every problem $L$, there exists $\eps > 0$, such that for any sequence
$\tilde \Lambda$, satisfying the estimate $\rho(\Lambda, \tilde \Lambda) \le \eps$, there exists such numeration in the
sequences $\{ \pm \mu_n \}_{n \ge 0}$, $\{ \pm \nu_n \}_{n \ge 0}$, $\{ \pm \tilde \mu_n \}_{n \ge 0}$ and
$\{ \pm \tilde \nu_n \}_{n \ge 0}$, that
\begin{equation} \label{summu}
    \left( \sum_{n = 0}^{\iy} (n + 1)^2 |\mu_n - \tilde \mu_n|^2 \right)^{1/2} \le C \rho(\Lambda, \tilde \Lambda), \quad
    \left( \sum_{n = 0}^{\iy} (n + 1)^2 |\nu_n - \tilde \nu_n|^2 \right)^{1/2} \le C \rho(\Lambda, \tilde \Lambda).
\end{equation}
\end{lem}

\begin{proof}
We will prove the lemma for the sequences $\{ \mu_n \}_{n \ge 0}$ and $\{ \tilde \mu_n \}_{n \ge 0}$. The proof for
$\{ \nu_n \}_{n \ge 0}$ and $\{ \tilde \nu_n \}_{n \ge 0}$ is analogous.

Let the eigenvalues $\{ \mu_n^2 \}_{n \ge 0}$ of the problem $L_0$ be numbered in the increasing order: $\mu_n^2 < \mu_{n+1}^2$,
$n \ge 0$. The following asymptotic relation holds:
\begin{equation} \label{asymptmu}
    \mu_n = \pi(2n + 1) + O\left( n^{-1} \right), \quad n \to \iy.
\end{equation}

For simplicity, assume that $\mu_n > 0$, $n \ge 0$. The general case requires minor technical modifications.
For brevity, denote $\eta_1(\la) = \vv_1(1/2, \la)$, $\dot{\eta}_1 = \frac{d}{d\la} \eta_1(\la)$.
Choose such positive constants $r$ and $c_0$, that $r < \mu_0$, $r < \mu_{n+1} - \mu_n$, $|\dot{\eta}_1(\la)| \ge c_0$
for $|\la - \mu_n| \le r$, $n \ge 0$. Define the contours $\ga_{n, r} := \{\la\colon |\la - \mu_n| = r \}$, $n \ge 0$.
Obviously, $|\eta_1(\la)| \ge c_r > 0$ on $\ga_{n, r}$, where the constant $c_r$ depends on $r$, but does not depend on $n$.
In view of the relations~\eqref{2.5} and~\eqref{eta1}, we have the estimate
$$
    |\eta_1(\la) - \tilde\eta_1(\la)| \le \frac{C}{n + 1} \| \hat K_1 \|_{L_2}, \quad \la \in \ga_{n, r}, \quad n \ge 0,
$$
where $\hat K_1 := K_1 - \tilde K_1$, and the constant $C$ does not depend on $n$.

By virtue of Lemma~\ref{lem:stabR}, for every fixed problem $L$ and $\eps_0 > 0$, one can choose $\eps > 0$, such that,
for $\tilde \Lambda \in \mathbb M$, $\rho(\Lambda, \tilde \Lambda) \le \eps$, the estimate $\| \hat K_1 \|_{L_2} \le \eps_0$ holds.
If $\eps_0$ is chosen sufficiently small, we get $\dfrac{|\eta_1(\la) - \tilde \eta_1(\la)|}{|\eta_1(\la)|} < 1$ on $\ga_{n, r}$,
$n \ge 0$. Rouch\'{e}`s theorem implies that there is exactly one zero of $\tilde \eta_1(\la)$ inside each contour $\ga_{n, r}$,
$n \ge 0$, and we denote this zero by $\tilde \mu_n$. By using Rouch\'{e}`s theorem, one can also show that the function
$\tilde \eta_1(\la)$ has no other zeros except $\{ \pm \tilde \mu_n \}_{n \ge 0}$.

The following Taylor expansion is valid for every $n \ge 0$:
$$
    \eta_1(\tilde \mu_n) = \eta_1(\mu_n) + \dot{\eta}_1(\theta_n) (\tilde \mu_n - \mu_n), \quad \theta_n \in \mbox{int}\,\ga_{n, r}.
$$
Using this expansion together with~\eqref{2.5} and~\eqref{eta1}, we obtain
$$
    \eta_1(\tilde \mu_n) - \tilde \eta_1(\tilde \mu_n) = \frac{1}{\tilde \mu_n} \int_0^{1/2} \hat K_1(x) \sin (\tilde \mu_n x) \, dx
     = \dot \eta_1(\theta_n) (\tilde \mu_n - \mu_n).
$$
Note that $|\eta_1(\theta_n)| \ge c_0$, since $\theta_n \in \mbox{int}\,\ga_{n, r}$. Furthermore, for $\tilde \mu_n$
the asymptotic formula similar to~\eqref{asymptmu} is valid. Consequently, $\sin \tilde \mu_n = \sin ((2n + 1) \pi x) + O\left(n^{-1}\right)$ as $n \to \iy$.
Summarizing the arguments above, we arrive at the estimate
\begin{equation} \label{difmu}
    |\mu_n - \tilde \mu_n| \le \frac{C}{n + 1} \hat k_n + \frac{C}{(n + 1)^2} \| \hat K_1 \|_{L_2},
\end{equation}
where $\{ \hat k_n \}_{n \ge 0}$ are the Fourier coefficients of $\hat K_1(x)$:
$$
    \hat k_n := \int_0^{1/2} \hat K_1(x) \sin ((2n + 1) \pi x) \, dx, \quad n \ge 0.
$$
Using~\eqref{difmu} and Bessel inequality for $\{ \hat k_n \}_{n \ge 0}$, we get
$$
   \left(  \sum_{n = 0}^{\iy} (n + 1)^2 |\mu_n - \tilde \mu_n|^2 \right)^{1/2} \le C \| \hat K_1 \|_{L_2}.
$$
In view of Lemma~\ref{lem:stabR}, the latter estimate yields the first inequlity in \eqref{summu}.
\end{proof}

Lemma~\ref{lem:zeros} together with
the relations~\eqref{eta1}, \eqref{eta2} imply that,
for sufficiently small $\eps > 0$ and $\rho(\Lambda, \tilde \Lambda) \le \eps$, the numbers
$\{ \tilde \mu_n^2 \}_{n \ge 0}$ and $\{ \tilde \nu_n^2 \}_{n \ge 0}$ are real and distinct, as well as
$\{ \mu_n^2 \}_{n \ge 0}$ and $\{ \nu_n^2 \}_{n \ge 0}.$
Therefore one can apply Borg's Theorem in the following form (see \cite[Theorem~1.8.1]{FY01}).

\begin{thm}[Borg] \label{thm:Borg}
For the boundary value problems $L_j$, $j = 0, 1$, of the form~\eqref{Lj}, there exists $\eps > 0$
(which depends on $L_j$) such that if real numbers $\{ \tilde \mu_n^2 \}_{n \ge 0}$ and $\{ \tilde \nu_n^2 \}_{n \ge 0}$
satisfy the condition
$$
    \Omega := \left( \sum_{n = 0}^{\iy} (n + 1)^2 (\mu_n - \tilde \mu_n)^2 \right)^{1/2} +
    \left( \sum_{n = 0}^{\iy} (n + 1)^2 (\nu_n - \tilde \nu_n)^2 \right)^{1/2} \le \eps,
$$
then there exists a unique real pair $\tilde q_1(x) \in L_2(0, 1/2)$ and $\tilde h_1$, for which the numbers
$\{ \tilde \mu_n^2 \}_{n \ge 0}$ and $\{ \tilde \nu_n^2 \}_{n \ge 0}$ are the eigenvalues of the problems $\tilde L_j$, $j = 0, 1$,
respectively. The problems $\tilde L_j$ have the form~\eqref{Lj}, but with $\tilde q_1$ instead of $q_1$ and $\tilde h_1$
instead of $h_1$. Moreover,
$$
    \| q_1 - \tilde q_1 \|_{L_2} \le C \Omega, \quad |h_1 - \tilde h_1| \le C \Omega,
$$
where the constant $C$ depends only on $L_j$, $j = 0, 1$.
\end{thm}

\begin{proof}[Proof of Theorem~\ref{thm:loc}]
According to Lemma~\ref{lem:zeros} and Theorem~\ref{thm:Borg}, for every problem $L$ there exists $\eps > 0$, such that
for any $\tilde \Lambda \in \mathbb M$, satisfying the estimate $\rho(\Lambda, \tilde \Lambda) < \eps$,
there exist a real function $\tilde q_1(x) \in L_2(0, 1/2)$ and a real number $\tilde h_1$, such that $\tilde \eta_j(\la)$
is the characteristic function of $L_{j-1}$, $j = 1, 2$. Moreover, the estimates~\eqref{estq} are valid for these $\tilde q_1$
and $\tilde h_1$. In order to finish the proof of Theorem~\ref{thm:loc}, we have to show that the spectrum of the problem
$\tilde L = L(1/2, \tilde q_1, q_2, \tilde h_1, h_2, a_1, a_2)$ coincides with $\{ \tilde \la_n^2 \}_{n \ge 0}$.

Denote by $\tilde \vv(x, \la)$ the solution of the initial value problem
$$
    -\tilde \vv'' + \tilde q_1(x) \tilde\vv = \lambda^2 \tilde\vv, \quad \tilde \vv(0, \la) = 1, \quad \tilde \vv'(0, \la) = \tilde h_1.
$$
Then $\tilde \vv(1/2, \la) \equiv \tilde \eta_1(\la)$, $\tilde \vv'(1/2, \la) \equiv \tilde \eta_2(\la)$. Therefore characteristic
function of $\tilde L$ has the form
$$
    \tilde \Delta(\la) = a_1 \tilde \eta_1(\la) \vv_2'(1/2, \la) + a_1^{-1} \tilde \eta_2(\la) \vv_2(1/2, \la) +
    a_2 \tilde \eta_1(\la) \vv_2(1/2, \la).
$$
By construction, we have $(\tilde K, \tilde v_n)_H = \tilde f_n$, $n \ge 0$. In other words, the following relation holds
\begin{multline*}
 \left( \tilde \la_n \cos \frac{\tilde \la_n}{2} + \om_1 \sin \frac{\tilde \la_n}{2} +
\int_0^{1/2} \tilde K_1(x) \sin (\tilde \la_n x) \, dx \right) \cdot \frac{1}{\tilde \la_n} (a_1 \vv_2'(1/2, \tilde \la_n)
+ a_2 \vv_2(1/2, \tilde \la_n))
\\ + \left(-\tilde \la_n \sin \frac{\tilde \la_n}{2} + \om_1 \cos \frac{\tilde \la_n}{2} + \int_0^{1/2} \tilde K_2(x)
\cos (\tilde \la_n x) \, dx \right) \cdot a_1^{-1} \vv_2(1/2, \tilde \la_n) = 0, \quad n \ge 0.
\end{multline*}
In view of the definitions~\eqref{eta1},~\eqref{eta2}, we obtain $\tilde \Delta(\tilde \la_n) = 0$, $n \ge 0$. Therefore
$\{ \tilde \la_n^2 \}_{n \ge 0}$ are the eigenvalues of the constructed problem $\tilde L$.
\end{proof}

\bigskip

{\large \bf 4. The case $0 < d < 1/2$}

\bigskip

In this section, Inverse Problem~\ref{ip:2} is studied.
First of all, we formulate a uniqueness theorem. Let $I$ be a fixed subset of $\mathbb N \cup \{ 0 \}$.
Suppose that $\{ \la_n^2 \}_{n \ge 0}$ and $\{ \tilde \la_n^2 \}_{n \ge 0}$ are the eigenvalues of the boundary value problems
$L = L(d, q_1, q_2, h_1, h_2, a_1, a_2)$ and $\tilde L = L(d, \tilde q_1, q_2, \tilde h_1, h_2, a_1, a_2)$, respectively.

\begin{thm} \label{thm:uniq}
Suppose that $\la_n = \tilde \la_n$, $n \in I$, $\om_1 = \tilde \om_1$, and the system $\{ \exp(\pm i \la_n x) \}_{n \in I}$ is complete in
$L_2(-2d, 2d)$. Then $q_1 = \tilde q_1$ in $L_2(0, d)$ and $h_1 = \tilde h_1$.
\end{thm}

Theorem~\ref{thm:uniq} is an immediate corollary of the main equation~\eqref{scal} for $n \in I$
and the following Lemma.

\begin{lem} \label{lem:compI}
Suppose that the system $\{ \exp(\pm i \la_n x) \}_{n \in I}$ is complete in $L_2(-2d, 2d)$.
Then the system $\{ v_n \}_{n \in I}$, constructed by~\eqref{defv}, is complete in $H$.
\end{lem}

\begin{proof}
Let $w_1$ and $w_2$ be functions from $L_2(0, d)$, not both equal zero and satisfying the relation
\begin{multline} \label{sm1}
    \int_0^d \biggl(w_1(x) \frac{1}{\la_n} (a_1 \vv_2'(1-d,\la_n) + a_2 \vv_2(1-d,\la_n)) \sin (\la_n x) \\ + w_2(x) a_1^{-1} \vv_2(1-d,\la_n)
    \cos (\la_n x) \biggr) \, dx = 0, \quad n \in I. 	
\end{multline}
In view of~\eqref{char} and the equality $\Delta(\la_n) = 0$, the condition~\eqref{sm1} is equivalent to the following one:
$$
     \int_0^d \left( w_1(x) \vv_1'(d, \la_n) \frac{\sin (\la_n x)}{\la_n} - w_2(x) \vv_1(d, \la_n) \cos (\la_n x)\right) \, dx = 0,
     \quad n \in I.
$$
Hence the entire function
$$
    W(\la) := \int_0^d \left( w_1(x) \vv_1'(d, \la) \frac{\sin (\la x)}{\la} - w_2(x) \vv_1(d, \la) \cos (\la x)\right) \, dx
$$
has the zeros $\{ \pm \la_n \}_{n \in I}$. Clearly, $W \in \mathcal L^{2d}$.
Since the system $\{ \exp(\pm i \la_n x) \}_{n \in I}$ is complete in $L_2(-2d, 2d)$, Paley-Wiener Theorem yields $W(\la) \equiv 0$.
Consequently, one can easily show that $w_1 = w_2 = 0$ in $L_2(0, d)$, so the system $\{ v_n \}_{n \in I}$ is complete in $H$.
\end{proof}

Now suppose that $\{ v_n \}_{n \in I}$ is a Riesz basis in $H$. Obviously, in this case, the solution of Inverse Problem~2 is unique.
Moveover, one can construct this solution, by using the following algorithm.

\begin{alg} \label{alg:2}
Suppose that $0 < d < 1/2$. Let $\{ \la_n^2 \}_{n \ge 0}$, $q_2$, $h_2$, $a_1$, $a_2$ and $\om_1$ be given. We have to construct
$q_1$ and $h_1$.
\begin{enumerate}
\item Construct the functions $\vv_2(1-d, \la)$ and $\vv_2'(1-d, \la)$, using $q_2$ and $h_2$.
\item Find the vector-functions $\{ v_n(x) \}_{n \in I}$ and the numbers $\{ f_n \}_{n \in I}$ by~\eqref{defv} and~\eqref{deff},
respectively.
\item Determine the vector-function $K(x)$, using its coordinates with respect to the Riesz basis (see \eqref{scal}):
$$
    K(x) = \sum_{n \in I} f_n v_n^*(x),
$$
where $\{ v_n^* \}_{n \in I}$ is the Riesz basis, biorthonormal to $\{ v_n \}_{n \in I}$.
\item Similarly to the steps~7-8 of Algorithm~\ref{alg:1}, use $K(x)$ to construct $q_1$ and $h_1$.
\end{enumerate}
\end{alg}

Using Algorithm~\ref{alg:2}, the following theorem on local solvability and stability of Inverse Problem~\ref{ip:2} can be proved.

\begin{thm} \label{thm:locI}
Let $L = L(d, q_1, q_2, h_1, h_2, a_1, a_2)$ be any fixed problem, and let the subsequence of its eigenvalues $\{ \la_n^2 \}_{n \ge 0}$
be such that the system of vector-functions $\{ v_n \}_{n \in I}$, constructed by~\eqref{defv}, be a Riesz basis in $H$.
Then there exists $\eps > 0$, such that for any sequence
$\tilde \Lambda_I = \{ \tilde \la_n \}_{n \in I}$, satisfying the estimate
$$
    \rho_I := \left( \sum_{n \in I} \la_n^2 (\la_n - \tilde \la_n)^2 \right)^{1/2} \le \eps,
$$
there exist a function $\tilde q_1 \in L_2(0, d)$ and a number $\tilde h_1 \in \mathbb R$, such that $\{\tilde \la_n^2 \}_{n \in I}$
are eigenvalues of the problem $\tilde L = L(d, \tilde q_1, q_2, \tilde h_1, h_2, a_1, a_2)$,
and
\begin{equation*}
   \| q_1 - \tilde q_1 \|_{L_2} \le C \rho_I, \quad |h_1 - \tilde h_1| \le C \rho_I,
\end{equation*}
where the constant $C > 0$ depends only on the problem $L$ and does not depend on $\tilde \Lambda_I$.
\end{thm}

Theorem~\ref{thm:locI} shows that the subspectrum $\{ \la_n^2 \}_{n \in I}$ is minimal data, uniquely specifying
$q_1(x)$ and $h_1$, if $\{ v_n \}_{n \in I}$ is a Riesz basis. Indeed, Theorem~\ref{thm:locI} claims existence of the
inverse problem solution under any small perturbation of $\{ \la_n^2 \}_{n \in I}$.
In the case when Inverse Problem~\ref{ip:2}
is overdetermined, a small perturbation of the subspectrum can lead to absence of solution.

\begin{example}
Consider $d = \frac{1}{4}$. In this case, the eigenvalues of the problem $L$ fulfill the asymptotic formula
$$
    \la_n = \la_n^0 + O\left( \frac{1}{\la_n^0}\right), \quad n \ge 0, \: n \to \iy,
$$
where $\{ \pm \la_n^0 \}_{n \ge 0}$ are the zeros of the function
\begin{gather*}
    \Delta_0(\la) = \frac{\la}{2} \left((a_1 + a_1^{-1}) \sin \la + (a_1 - a_1^{-1}) \sin \frac{\la}{2}\right), \\
    \la_n^0 \ge 0, \quad \la_n^0 < \la_{n+1}^0, \quad n \ge 0.
\end{gather*}
In particular, $\la_{2n}^0 = 2 \pi n$, $n \ge 0$.

Set $I := \{ 2 n \colon n \in \mathbb N \cup \{ 0 \} \}$. Similarly to the proof of Theorem~\ref{thm:Riesz}, one can show,
that the system $\{ v_n \}_{n \in I}$, constructed by~\eqref{defv}, is a Riesz basis in $H = L_2(0, 1/4) \oplus L_2(0, 1/4)$.
Indeed, the completeness of $\{ v_n \}_{n \in I}$ can be proved similarly to Lemmas~\ref{lem:complete} and~\ref{lem:compI}.
Moreover, the following asymptotic relation holds:
\begin{gather}
    v_{2n}(x) = v_{2n}^0(x) + O\left( n^{-1} \right), \quad n \to \iy, \\
    v_{4k}^0(x) = (-1)^k a_1^{-1} \begin{pmatrix}
                      0 \\ \cos (4 k \pi x)
                  \end{pmatrix}, \\
    v_{4k + 2}^0(x) = (-1)^k a_1 \begin{pmatrix}
                       \sin ((4 k + 2) \pi x) \\
                       0
                   \end{pmatrix}, \quad k \ge 0.
\end{gather}
Clearly, $\{ v_{2n}^0 \}_{n \ge 0}$ is a Riesz basis in $H$, so the same is true for $\{ v_{2n}(x) \}_{n \ge 0}$.
Consequently, the eigenvalues $\{ \la_{2n} \}_{n \ge 0}$ uniquely specify the solution of Inverse Problem~\ref{ip:2},
which can be constructed by Algorithm~\ref{alg:2}. Theorem~\ref{thm:locI} is applicable to this case, so local solvability and
stability of solution hold.
\end{example}

\medskip

{\bf Acknowledgment.} The author C.F.~Yang was supported in part by the National Natural Science Foundation of China
(11871031 and 11611530682). The author N.P.~Bondarenko was supported by Grant 1.1660.2017/4.6 of the Russian
Ministry of Education and Science and by Grant 19-01-00102 of the Russian Foundation for Basic Research.

\medskip

\noindent Chuan-Fu Yang \\
Department of Applied Mathematics, Nanjing University of Science and Technology, \\
Nanjing, 210094, Jiangsu, China, \\
email: {\it chuanfuyang@njust.edu.cn}

\medskip

\noindent Natalia Pavlovna Bondarenko \\
1. Department of Applied Mathematics and Physics, Samara National Research University, \\
Moskovskoye Shosse 34, Samara 443086, Russia, \\
2. Department of Mechanics and Mathematics, Saratov State University, \\
Astrakhanskaya 83, Saratov 410012, Russia, \\
e-mail: {\it BondarenkoNP@info.sgu.ru}


\medskip

\begin{thebibliography}{99}

\bibitem{FY01}
Freiling, G.; Yurko, V. Inverse Sturm-Liouville Problems and Their Applications. Huntington,
NY: Nova Science Publishers, 2001.

\bibitem{And97}
Anderssen, R.S. The effect of discontinuities in density and shear velocity on the
asymptotic overtone structure of tortional eigenfrequencies of the Earth, Geophys.
J.R. Astr. Soc. 50 (1997), 303--309.

\bibitem{LU81}
Lapwood, F.R.; Usami, T. Free Oscillations of the Earth, Cambridge University
Press, Cambridge, 1981.

\bibitem{LS64}
Litvinenko, O.N.; Soshnikov, V.I. The Theory of Heterogenious Lines and their
Applications in Radio Engineering, Moscow: Radio, 1964 (Russian).

\bibitem{MF80}
Meschanov, V.P.; Feldstein, A.L. Automatic Design of Directional Couplers, Moscow:
Sviaz, 1980 (Russian).

\bibitem{Mar77}
Marchenko, V. A. Sturm-Liouville Operators and their Applications, Naukova Dumka,
Kiev (1977) (Russian); English transl., Birkhauser (1986).

\bibitem{Lev84}
Levitan, B. M. Inverse Sturm-Liouville Problems, Nauka, Moscow (1984) (Russian); English
transl., VNU Sci. Press, Utrecht (1987).

\bibitem{PT87}
P\"{o}schel, J; Trubowitz, E. Inverse Spectral Theory, New York, Academic Press (1987).

\bibitem{Hald84}
Hald O. Discontinuous inverse eigenvalue problem, Commun. Pure Appl. Math. (1984), Vol. 37, 53--577.

\bibitem{SY08}
Shieh, C.-T.; Yurko, V.A. Inverse nodal and inverse spectral problems for discontinuous boundary value problems,
J. Math. Anal. Appl. 347 (2008), 266--272.

\bibitem{OK12}
Ozkan A.S., Keskin B. Uniqueness theorems for an impulsive Sturm-Liouville boundary value problem, Appl. Math. J. Chinese Univ.
27:4 (2012), 428--434.

\bibitem{Yang14-1}
Yang, C.-F. Inverse problems for the Sturm-Liouville operator with discontinuity, Inverse Problems in Science
and Engineering 22: 2 (2014), 232--244.

\bibitem{Wang15}
Wang, Y.P. Inverse problems for discontinuous Sturm-Liouville operators with mixed spectral data,
Inverse Problems in Science and Engineering 23:7 (2015), 1180--1198.

\bibitem{AKM96}
Aktosun, T.; Klaus, M.; van der Mee, C. Recovery of discontinuities in a nonhomogeneous
medium, Inverse Problems, 12 (1996), 1-25.

\bibitem{Shep94}
Shepelsky D.G., The inverse problem of reconstruction of the medium‘s conductivity in
a class of discontinuous and increasing functions, Spectral operator theory and related
topics, 209-232, Advances in Soviet Math., 19, Amer. Math. Soc., Providence, RI, 1994.

\bibitem{HL78}
Hochstadt, H.; Lieberman, B. An inverse Sturm-Liouville problem with mixed given data, SIAM J. Appl. Math. 34: 4 (1978), 676--680.

\bibitem{GS00}
Gesztesy, F.; Simon, B. Inverse spectral analysis with partial information on the potential, II. The case of discrete spectrum,
Trans. AMS 352:6 (2000), 2765--2787.

\bibitem{Hor01}
Horvath, M. On the inverse spectral theory of Schr\"odinger and Dirac operators, Trans. AMS 353: 10 (2001), 
4155--4171.
	
\bibitem{Sakh01}
Sakhnovich, L. Half-inverse problems on the finite interval, Inverse Problems 17 (2001), 527--532.

\bibitem{HM04}
Hryniv, R.O.; Mykytyuk, Ya. V.  Half-inverse spectral problems for Sturm-Liouville
operators with singular potentials, Inverse Problems 20 (2004), 1423--1444.

\bibitem{Piv12}
Pivovarchik, V. On the Hald-Gesztesy-Simon theorem, Integral Equations
and Operator Theory 73 (2012), 383--393.

\bibitem{SBI11}
Shieh, C.-T.; Buterin, S.A.; Ignatiev, M. On Hochstadt-Lieberman theorem for Sturm-Liouville operators.
Far East Journal of Applied Mathematics 52 (2011), 131--146.

\bibitem{But11}
Buterin, S.A. On half inverse problem for differential pencils with the spectral parameter in boundary conditions,
Tamkang J. Math. 42 (2011), 355--364.

\bibitem{BY19}
Yang, C.-F.; Bondarenko, N.P. Reconstruction and solvability for discontinuous Hochstadt-Lieberman problems,
J. Spectral Theory (in Press), 	arXiv:1904.10263 [math.SP].

\bibitem{Bond18}
Bondarenko, N. P. A partial inverse problem for the Sturm-Liouville operator on a star-shaped graph, Anal. Math. Phys. 8: 1 (2018), 155--168.

\bibitem{Yang14}
Yang, C.-F. Traces of Sturm-Liouville operators with discontinuities, Inverse Problems in Science and
Engineering 22 (2014), 803--813.

\bibitem{BFY}
Buterin, S. A.; Freiling, G.; Yurko, V. A. Lectures in the theory of entire functions, Schriftenriehe der Fakult\"at f\"ur Matematik, Duisbug-Essen University, SM-UDE-779 (2014).

\bibitem{BB17}
Bondarenko N.; Buterin S. On a local solvability and stability of the inverse transmission
eigenvalue problem, Inverse Problems 33 (2017), 115010.

\end{thebibliography}
\end{document}